\documentclass{amsart}
\usepackage{amssymb,amsmath,amsthm,color,graphicx,mathrsfs}
\DeclareMathOperator{\charac}{char}
\DeclareMathOperator{\dimirr}{dimirr}
\DeclareMathOperator{\Stab}{Stab}
\DeclareMathOperator{\Gr}{Gr}
\DeclareMathOperator{\Irr}{Irr}
\DeclareMathOperator{\Res}{Res}
\DeclareMathOperator{\Ind}{Ind}

\DeclareMathOperator{\Hom}{Hom}

\title[A Polynomial Result for Dimensions of Representations]{A Polynomial Result for Dimensions of Irreducible Representations of Smooth Affine Group Schemes Over Principal Ideal Local Rings}
\date{\today}
\author{Alexander Jackson}
\newtheorem{theorem}{Theorem}
\newtheorem{lemma}[theorem]{Lemma}
\newtheorem{corollary}[theorem]{Corollary}
\newtheorem{proposition}[theorem]{Proposition}
\newtheorem{conjecture}[theorem]{Conjecture}
\newtheorem{example}[theorem]{Example}
\theoremstyle{definition}
\newtheorem{definition}[theorem]{Definition}
\begin{document}

\begin{abstract}
Denote by $\mathfrak{o}$ the valuation ring of a non-Archimedean local field with prime ideal $\mathfrak{p}$ and finite residue field, and let $r\geq 1$ be an integer. We prove that for every smooth affine group scheme $G$ over $\mathbb{Z}$, the dimension of each irreducible representation of $G(\mathfrak{o}/\mathfrak{p}^r)$ is given by one of finitely many polynomials with coefficients in $\mathbb{Q}$ evaluated at $q=|\mathfrak{o}/\mathfrak{p}|$, provided that the residue characteristic $p=\charac\mathfrak{o}/\mathfrak{p}$ is large and fixed.
\end{abstract}
\maketitle
\begin{figure}[b]
    \makebox[\textwidth][l]{
        \includegraphics[width=0.3\textwidth]{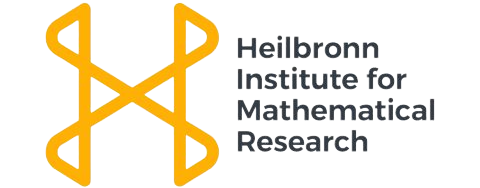}
        \includegraphics[width=0.15\textwidth]{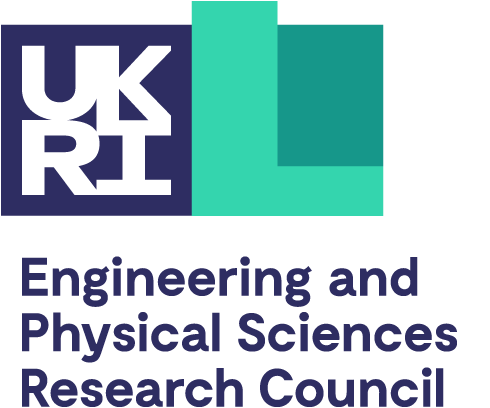}
    }
\end{figure}
\section{Introduction}
For every non-Archimedean local field $F$ with finite residue field, we define 
$\mathfrak{o}$ to be the valuation ring of $F$,
$\mathfrak{p}$ the unique non-zero prime ideal of $\mathfrak{o}$,
$q=|\mathfrak{o}/\mathfrak{p}|$ and
$p=\charac\mathfrak{o}/\mathfrak{p}$.
Throughout, we fix integers $n,r\geq 1$, and we write $\mathfrak{o}_r=\mathfrak{o}/\mathfrak{p}^r$. Given a group, $G$, we denote by $\Irr(G)$ the set of irreducible complex representations of $G$ up to isomorphism, and $\dimirr(G)=\{\dim\rho\mid\rho\in\Irr(G)\}$.
\par We concern ourselves with the problem of describing the smooth representations of $\mathrm{GL}_n(\mathfrak{o})$, or equivalently, those which factor through some $\mathrm{GL}_n(\mathfrak{o}_r)$. Green \cite{green} constructed all the representations of $\mathrm{GL}_n(\mathbb{F}_q)$, hence finding all representations in the case $r=1$. In a series of papers, Hill \cite{hilljordan,hillnilpotent,hill95,hillsemisimplecuspidal} developed a method to construct the representations of $\mathrm{GL}_n(\mathfrak{o}_r)$ based on Clifford Theory (see Section~ \ref{Clifford Theory}), in which one chooses an abelian congruence kernel $1+\mathfrak{p}^{i}\mathrm{M}_n(\mathfrak{o}_r)$ as the chosen normal subgroup (here, $i\geq\lceil\frac r2\rceil$). Later, Stasinski \cite{stasinski08} described the representations of~$\mathrm{GL}_2(\mathfrak{o}_r)$ and Singla \cite{singla10} described the representations of~$\mathrm{GL}_n(\mathfrak{o}_2)$ for $n\leq 4$. In a more recent preprint, Onn, Prasad and Singla \cite{onnprasadsingla23} have constructed the representations of $\mathrm{GL}_3(\mathfrak{o}_r)$. Krakovski, Onn and Singla \cite{krakovskionnsingla} classified all the regular representations of $\mathrm{GL}_n(\mathfrak{o}_r)$ (as defined by Hill in \cite{hill95}) where $p\neq 2$, and in independent work, Stasinski and Stevens \cite{stasinskistevens} did the same for arbitrary $p$.
\par Onn \cite{onn08} made the following conjecture:
\begin{conjecture}\label{onn weak conjecture}
If $F,F'$ are non-Archimedean local fields with rings of integers $\mathfrak{o},\mathfrak{o}'$ of the same finite residue cardinality $q$, there is an isomorphism of group algebras $\mathbb{C}[\mathrm{GL}_n(\mathfrak{o}_r)]\cong\mathbb{C}[\mathrm{GL}_n(\mathfrak{o}'_r)]$, or equivalently, a dimension-preserving bijection
\[\Irr(\mathrm{GL}_n(\mathfrak{o}_r))\longleftrightarrow \Irr(\mathrm{GL}_n(\mathfrak{o}'_r)).\]
\end{conjecture}
Also in \cite{onn08}, Onn conjectured that the dimensions of the representations, along with their frequencies, are given by polynomials in $q$. We suggest the following more precise restatement of these conjectures:
\begin{conjecture}\label{onn conjecture}
Let $\mathfrak{O}$ denote the set of rings which are the valuation ring of a non-Archimedean local field, up to isomorphism. Let $n,r\geq 1$ be given. There exist $k\geq 1$, polynomials \[d_1(x),\dots,d_k(x)\in\mathbb{Z}[x]\setminus\{0\},m_1(x),\dots,m_k(x)\in\mathbb{Q}[x]\setminus\{0\},\]
and for all $\mathfrak{o}\in\mathfrak{O}$ and $i\in\{1,\dots,k\}$, there exist subsets $\mathcal{R}_{i,\mathfrak{o}}$ of $\Irr(\mathrm{GL}_n(\mathfrak{o}_r))$
such that $\bigcup_{i=1}^k\mathcal{R}_{i,\mathfrak{o}}=\Irr(\mathrm{GL}_n(\mathfrak{o}_r))$ is a disjoint union,
and
\begin{enumerate}
\item $\text{for all }\rho\in\mathcal{R}_{i,\mathfrak{o}},\dim\rho=d_i(|\mathfrak{o}_1|)$, and
\item $|\mathcal{R}_{i,\mathfrak{o}}|=m_i(|\mathfrak{o}_1|)$.
\end{enumerate}
\end{conjecture}
A \textbf{twist} of a representation is the tensor product with a one-dimensional representation. A representation of $\mathrm{GL}_n(\mathfrak{o}_r)$ is \textbf{twist primitive} if none of its twists factor through the quotient $\mathrm{GL}_n(\mathfrak{o}_{r-1})$. One can also modify Conjecture \ref{onn conjecture} by replacing $\Irr(\mathrm{GL}_n(\mathfrak{o}_r))$ with the set of twist primitive representations of $\mathrm{GL}_n(\mathfrak{o}_r)$. In this case, \cite[Conjecture 1.5]{onn08} predicts that the $d_i$ can be chosen to have degree at most $\binom n2 r$.
Note that we allow for the possibility that some $\mathcal{R}_{i,\mathfrak{o}}$ is empty, as the following example shows:
\begin{example}
Set $n=2,r=1$, that is, we look at the representations of $GL_2(\mathbb{F}_q)$. It is well known that we require $k=4$ and the following polynomials:
\begin{center}
\begin{tabular}{c|c|c}
$i$&$d_i(x)$&$m_i(x)$\\\hline
$1$&$1$&$x-1$\\
$2$&$x-1$&$\frac 12 x(x-1)$\\
$3$&$x$&$x-1$\\
$4$&$x+1$&$\frac 12(x-1)(x-2)$
\end{tabular}.
\end{center}
We see that $m_4(2)=0$, i.e. $\mathcal{R}_{4,\mathfrak{o}}=\emptyset$ whenever $\mathfrak{o}$ has residue cardinality equal to $2$.
\end{example}
\par Further in \cite{onn08}, Onn constructed the representations of $\mathrm{Aut}_\mathfrak{o}(\mathfrak{o}_{r_1}\oplus\mathfrak{o}_{r_2})$ and thus proved Conjecture \ref{onn conjecture} for $n=2$. Singla \cite{singla10} proved Conjecture \ref{onn weak conjecture} for $r=2$ and Conjecture \ref{onn conjecture} in the case when $r=2$ and $n\leq 4$; note that Green \cite{green} had already proven Conjecture \ref{onn conjecture} for $r=1$ by giving an explicit formula for dimensions of irreducible representations of $\mathrm{GL}_n(\mathbb{F}_q)$. Stasinski and Vera-Gajardo \cite{sv-g19} proved Conjecture \ref{onn weak conjecture} for $r=2$, where $\mathrm{GL}_n$ is replaced by an arbitrary reductive group scheme $G$ and $p$ is a very good prime for $G$. Hadas \cite{hadas24} proved Conjecture \ref{onn weak conjecture} for all~$r\geq 1$, all smooth affine group schemes $G$ in place of $\mathrm{GL}_n$, and local fields $F$ that are unramified over $\mathbb{Q}_p$ or $\mathbb{F}_p((t))$ and where $p$ is a large enough prime.
\par We note that Conjecture \ref{onn conjecture} does not hold if $\mathrm{GL}_n$ is replaced by an arbitrary smooth affine group scheme. Indeed, this is false for $\mathrm{SL}_3$, as it was shown by Simpson and Frame \cite{simpsonframe73} that $\mathrm{SL}_3$ has a class of irreducible representations whose dimension depends on the congruence class of $q$ modulo $3$. Moreover, Halasi and P\'alfy \cite{halasipalfy11} gave an example of a so-called ``pattern group" $G$ such that the number of conjugacy classes of $G(\mathbb{F}_q)$ is not given by a polynomial in $q$. However, the question of whether the number of conjugacy classes of the unitriangular group $\mathrm{U}_n(\mathbb{F}_q)$ is given by a polynomial in $q$ is still open; this is a well-known conjecture of Higman \cite{higman60part1}.
\par In this paper, we prove a modified version of Conjecture \ref{onn conjecture} for dimensions only, namely, that there exist finitely many polynomials which when evaluated at $q$, contain all the dimensions of the representations of $\mathrm{GL}_n(\mathfrak{o}_r)$, where $p$ is assumed to be large enough. In fact our method of proof, which is based on work of Hadas \cite{hadas24}, requires us to consider arbitrary linear algebraic groups, not only $\mathrm{GL}_n$. More precisely, we show the following:
\begin{theorem}[Weak Polynomial Conjecture for large $p$]\label{weak polynomial for large p}
Let $G$ be a smooth affine group scheme of finite type over $\mathbb{Z}$ and $r\geq 1$ an integer. There exists $C>0$ such that for every prime $p>C$, there exists a finite set of polynomials $R\subseteq \mathbb{Q}[x]$ such that for every non-Archimedean local field $F$ with valuation ring $\mathfrak{o}$ and finite residue field of characteristic $p$,
\[\mathrm{dimirr}(G(\mathfrak{o}_r))\subseteq\left\{m(|\mathfrak{o}_1|)\mid m\in R\right\}.\]
\end{theorem}
\par The structure of the paper is as follows: Section~ \ref{Clifford Theory} recalls preliminaries on Clifford Theory. Suppose $G$ is an algebraic group defined over a finite field $\mathbb{F}$; we will use Clifford Theory to describe representations of $G(\mathbb{F})$ using those of $R_u(G)(\mathbb{F})$. In Section~ \ref{Greenberg Functor}, we state the results we need about the Greenberg functor introduced in \cite{greenberg61}; we refer to \cite{bertapellegonzalezaviles18} for details. In Section~ \ref{Finite Groups of Lie Type}, we outline an argument of Geck \cite{geck11} which gives a weak polynomial result (in the sense of Theorem \ref{weak polynomial for large p}) for reductive group schemes. In Section~ \ref{Proof of the Main Theorem}, we prove an analogue of Theorem \ref{weak polynomial for large p}, where $F$ ranges over unramified extensions of some fixed base field $F_0$ (Corollary \ref{Weak Onn for towers}), and deduce Theorem \ref{weak polynomial for large p} in full generality.
\section{Clifford Theory}\label{Clifford Theory}
Let $G$ be a finite group and $N$ a normal subgroup of $G$. Clifford \cite{clifford37} described how one can build representations of $G$ by parameterising the representations $\psi$ of $N$, and then describing the representations of $G$ which lie above each $\psi$. Later, we shall apply this to the case where $G$ is the $\mathbb{F}$-points of an algebraic group over a finite field $\mathbb{F}$, and $N$ is the $\mathbb{F}$-points of its unipotent radical.
\par We fix some notation. For $\psi\in\Irr(N)$, write \[\Irr(G\mid\psi)=\{\rho\in\Irr(G)\mid\langle\psi,\mathrm{Res}^G_N(\rho)\rangle\neq 0\}.\]
There is an action of $G$ on $\Irr(N)$ by $({^g\chi})(n)=\chi(g^{-1}ng)$.
\par We summarise the results we require from Clifford Theory. Their proofs can be found in \cite[Theorems 6.2, 6.11, 11.7]{isaacs}.
\begin{theorem}[Clifford Theory]\label{clifford theorem}
Let $G$ be a finite group and $N$ a normal subgroup of $G$. Then the following hold:
\begin{enumerate}
\item For all $\rho\in\Irr(G)$, there exists an integer $e\geq 1$ and a $G$-orbit $\Omega\subseteq\Irr(N)$ such that
\[\mathrm{Res}^G_N(\rho)=e\bigoplus_{\psi\in\Omega}\psi.\]
\item There is a bijection
\[\Irr(\mathrm{Stab}_G(\psi)\mid\psi)\to\Irr(G\mid\psi)\]
given by
\[\theta\mapsto\mathrm{Ind}^G_{\mathrm{Stab}_G(\psi)}(\theta).\]
\item If $\psi$ extends to a representation $\hat{\psi}$ of $\mathrm{Stab}_G(\psi)$, there is a bijection
\[\Irr(\mathrm{Stab}_G(\psi)/N)\to\Irr(\mathrm{Stab}_G(\psi)\mid\psi)\]
given by
\[\overline{\theta}\mapsto\theta\otimes\hat{\psi},\]
where $\theta=\overline{\theta}\circ\pi$ and $\pi:\Stab_G(\psi)\to\Stab_G(\psi)/N$ is the natural map.
\end{enumerate}
\end{theorem}
In general, one can only say that $\psi$ extends to a \textit{projective representation} of $\Stab_G(\psi)$, which defines a cohomology class $[\alpha]\in H^2(\Stab_G(\psi)/N,\mathbb{C}^\times)$ (see e.g. \cite[Theorem 11.2]{isaacs}). In this case, $\psi$ extends to a representation of $\Stab_G(\psi)$ if and only if $[\alpha]$ is the trivial cohomology class. It is therefore sufficient to require that ${H^2(\Stab_G(\psi)/N,\mathbb{C}^\times)=1}$ in order for an extension to exist. Let $\mu_{p^\infty}$ denote the $p$-power roots of unity in $\mathbb{C}$. If $N$ is a normal $p$-subgroup of $G$ and ${H^2(\Stab_G(\psi)/N,\mu_{p^\infty})=1}$, then $\psi$ extends to a linear representation of $\Stab_G(\psi)$ by \cite[Lemma 9.0.2]{aizenbudavni22}.
\par Suppose $N$ is a normal $p$-subgroup of $G$ such that for every $\psi\in\Irr(N)$, we have $H^2(\Stab(\psi)/N,\mu_{p^\infty})=1$ (this is the situation we will encounter in the proof of Theorem \ref{hadas polynomial version}). Then $\psi$ extends to a representation $\hat{\psi}\in\Irr(\Stab_G(\psi))$ and Theorem \ref{clifford theorem} implies that every $\rho\in\Irr(G)$ satisfies one of the following two cases:
\par The restriction $\Res^G_N\rho$ is isotypic, meaning that $\Res^G_N\rho=e\psi$ for some $e\in\mathbb{N}$. We have $\Stab_G(\psi)=G$, and $\rho=\theta\otimes\hat{\psi}$, for some $\theta\in\Irr(G)$ which factors through $G/N$.
\par The restriction $\Res^G_N\rho$ is non-isotypic, and $\rho$ is induced from a representation of a proper subgroup of $G$.
\section{The Greenberg Functor}\label{Greenberg Functor}
Throughout this section, let $\mathfrak{o}$ be a complete discrete valuation ring with finite residue field $k$, and $\mathfrak{o}'$ an unramified extension of $\mathfrak{o}$ with residue field $k'$.
\par In \cite{greenberg61}, Greenberg constructed a functor, $\mathrm{Gr}^\mathfrak{o}_{r-1}$, mapping schemes of finite type over $\mathfrak{o}_r$ to schemes of finite type over $k$, with the property that \[\mathrm{Gr}^\mathfrak{o}_{r-1}(X)(\mathbb{F}_q)=X(\mathfrak{o}_r).\]
\begin{proposition}
If $X$ is smooth and affine of finite type over $\mathfrak{o}_r$, then $\Gr^\mathfrak{o}_{r-1}(X)$ is smooth and affine of finite type over $k$. Furthermore, $\Gr^\mathfrak{o}_{r-1}$ preserves fibre products, so if $X$ is a group scheme, then $\Gr^\mathfrak{o}_{r-1}(X)$ is also naturally a group scheme over $k$.
\end{proposition}
\begin{proof}
This result is the content of Corollary 1 of Proposition 4 and Corollary 4 of Proposition 6 in \cite{greenberg61}, and \cite[Corollary 1]{greenberg63}.
\end{proof}
We also require the following result of Nicaise and Sebag \cite[Proposition 3.3]{nicaisesebag08}:
\begin{theorem}\label{greenberg functor and extensions of o}
Let $X$ be a scheme of finite type over $\mathfrak{o}_r$. Observe that $\mathfrak{o}'_r$ is an $\mathfrak{o}_r$-algebra via the natural map $\mathfrak{o}_r\to\mathfrak{o}'_r$; this induces a isomorphism of $k'$-schemes:
\[\mathrm{Gr}^\mathfrak{o}_{r-1}(X)\times_kk'\cong\mathrm{Gr}^{\mathfrak{o}'}_{r-1}(X\times_{\mathfrak{o}_r}\mathfrak{o}'_r).\]
\end{theorem}
This result is stated only for finite extensions $\mathfrak{o}'/\mathfrak{o}$ in \cite{nicaisesebag08}, however the proof does not make use of the assumption that the extension is finite.
\section{Finite Groups of Lie Type}\label{Finite Groups of Lie Type}
We present an argument of Geck \cite{geck11} showing that reductive groups have the ``weak polynomial" property of Theorem \ref{weak polynomial for large p}, that is:
\begin{theorem}\label{weak polynomial for lie type}
Let $q$ be a power of a prime and $G$ be a connected reductive linear algebraic group over $\mathbb{F}_q$. There exists a finite set $R\subseteq\mathbb{Q}[x]$ of polynomials such that for every $d\geq 1$,
\[\dimirr(G(\mathbb{F}_{q^d}))\subseteq\{m(q^d)|m\in R\}.\]
\end{theorem}
Denote by $\overline{\mathbb{F}_q}$ the algebraic closure of $\mathbb{F}_q$. Specifying $G$ as in the statement of the theorem is equivalent to specifying a connected reductive group over $\overline{\mathbb{F}_q}$ with some $\mathbb{F}_{q^d}$-rational structure given by Frobenius morphism $F$; the fixed points of $F$ correspond to $G(\mathbb{F}_{q^d})$.
The following result will be useful in many places in subsequent proofs:
\begin{lemma}\label{frobenius factoring}
Let $p$ be a prime and $\mathbf{G}$ an algebraic group over $\overline{\mathbb{F}_p}$ with Frobenius map $F$, and $\mathbf{H}$ a connected $F$-stable algebraic subgroup. Then $\mathbf{G}^F/\mathbf{H}^F\cong(\mathbf{G}/\mathbf{H})^F$.
\end{lemma}
Throughout the rest of this section, bold face capital letters will denote linear algebraic groups over $\overline{\mathbb{F}_p}$. We fix a connected reductive group $\mathbf{G}$ with Weyl group $\mathbf{W}$, maximally split torus $\mathbf{T}_0$ and Frobenius map $F$.
\begin{definition}\cite[2.2]{geck17}
For $w\in\mathbf{W}$, define
\[\mathbf{T}_0[w]=\{t\in\mathbf{T}_0|F(t)=\dot{w}^{-1}t\dot{w}\},\]
where $\dot{w}$ is a lift of $w$ to $N_\mathbf{G}(\mathbf{T}_0)$. Where we need to make reference to the Frobenius map $F$, we write $\mathbf{T}_0[w,F]$.
\end{definition}
Deligne and Lusztig \cite{delignelusztig76} defined the virtual representations, $R^\theta_w$, of the finite group $\mathbf{G}^F$, indexed by $w\in\mathbf{W}$ and $\theta\in\Irr(\mathbf{T}_0[w])$. One defines a variety $X$ acted on the left by $\mathbf{G}^F$ and on the right by $\mathbf{T}_0[w]$; the $l$-adic cohomology groups of this variety are $\overline{\mathbb{Q}_l}[\mathbf{G}^F]-\overline{\mathbb{Q}_l}[\mathbf{T}_0[w]]$-bimodules which are finite-dimensional as $\overline{\mathbb{Q}_l}$-vector spaces. Define the formal alternating sum of these bimodules to be
\[H_c^*(X)=\sum_{i=0}^\infty (-1)^iH_c^i(X).\]
This sum is finite, since for all $i>2\dim X$, $H_c^i(X)=0$. For $\theta\in\Irr(\mathbf{T}_0[w])$, we can define $R_w^\theta=H_c^*(X)\otimes_{\mathbf{T}_0[w]}\theta$, which is a virtual representation of $\mathbf{G}^F$. We do not cover the construction in detail, instead we quote standard results in giving the argument.
\begin{proposition}[Inner product bound]\label{inner product bound}
Let $R^\theta_w,R^{\theta'}_{w'}$ be two Deligne-Lusztig representations. Then $\langle R^\theta_w,R^{\theta'}_{w'}\rangle$ is an integer and $0\leq \langle R^\theta_w,R^{\theta'}_{w'}\rangle\leq |\mathbf{W}|$.
\end{proposition}
\begin{proof}
This is a direct consequence of the inner product formula for Deligne-Lusztig representations, see e.g. \cite[Theorem 7.3.4]{carter93}.
\end{proof}
\begin{proposition}\label{uniformity of regular representation}\cite{geck17}
The character of the regular representation of $\mathbf{G}^F$ can be written as
\[\chi_{\mathrm{reg}}=\frac 1{|\mathbf{W}|}\sum_{w\in\mathbf{W}}\sum_{\theta\in\Irr(\mathbf{T}_0[w])}R^\theta_w(1)R^\theta_w.\]
\end{proposition}
\begin{corollary}\cite[Corollary 7.5.8]{carter93}
Every $\rho\in\Irr(\mathbf{G}^F)$ occurs in some Deligne-Lusztig representation, $R^\theta_w$. 
\end{corollary}
\begin{proof}
Write
\begin{align*}
\dim\rho&=\langle\rho,\chi_\mathrm{reg}\rangle\\
&=\frac 1{|\mathbf{W}|}\sum_{w\in\mathbf{W}}\sum_{\theta\in\Irr(\mathbf{T}_0[w])}R^\theta_w(1)\langle R^\theta_w,\rho\rangle.
\end{align*}
Then one of the terms in the sum must be non-zero, in particular $\langle R^\theta_w,\rho\rangle\neq 0$ for some $R^\theta_w$.
\end{proof}
\begin{definition}[Norm map]
Let $\mathbf{T}$ be an $F$-stable maximal torus in $\mathbf{G}$. The \textbf{$d$th norm map} is
\[N_{F^d/F}:\mathbf{T}^{F^d}\to\mathbf{T}^F;\quad t\mapsto tF(t)\dots F^{d-1}(t).\]
\end{definition}
\begin{definition}\label{geometrically conjugate}\cite[4.1.3]{carter93}
Let $\mathbf{T},\mathbf{T}'$ be $F$-stable maximal tori in $\mathbf{G}$ with associated norm maps $N_{F^d/F},N_{F^d/F}'$, and $\theta\in\Irr(\mathbf{T}^F),\theta'\in\Irr({\mathbf{T}'}^F)$. Then $(\mathbf{T},\theta)$ and $(\mathbf{T}',\theta')$ are \textbf{geometrically conjugate} if there exists $d\geq 1$ and $g\in\mathbf{G}^{F^d}$ such that ${\mathbf{T}'=g\mathbf{T}g^{-1}}$ and $g$ maps $\theta\circ N_{F^d/F}$ to $\theta'\circ N_{F^d/F}'$, i.e. for all $t\in(\mathbf{T}')^{F^d}$,
\[\theta'\circ N_{F^d/F}'(t)=\theta\circ N_{F^d/F}(g^{-1}tg).\]
\end{definition}
We say that two virtual representations, $R_1,R_2$ of a group $G$ are \textbf{disjoint} if for all $\rho\in\Irr(G)$, either $\langle R_1,\rho\rangle=0$ or $\langle R_2,\rho\rangle=0$. The following important result relates geometric conjugacy and disjointness of Deligne-Lusztig representations.
\begin{proposition}[Exclusion theorem]\label{exclusion}\cite[Theorem 7.3.8]{carter93}
Let $\mathbf{T}_0$ be a maximally split torus and $\mathbf{T},\mathbf{T}'$ be obtained by conjugating $\mathbf{T}_0$ with $g,g'\in\mathbf{G}$ respectively, so $\mathbf{T}^F=\mathbf{T}_0[w]$, where $w$ is the image of $g^{-1}F(g)$ in $\mathbf{W}$ and $(\mathbf{T}')^F=\mathbf{T}_0[w']$, where $w'$ is the image of $(g')^{-1}F(g')$ in $\mathbf{W}$. Let $\theta\in\Irr(\mathbf{T}_0[w]),\theta'\in\Irr(\mathbf{T}_0[w'])$. Suppose that $(\mathbf{T},\theta)$ and $(\mathbf{T}',\theta')$ are not geometrically conjugate. Then $R^\theta_w$ and $R^{\theta'}_{w'}$ are disjoint.
\end{proposition}
\begin{proposition}
Let $\rho$ be an irreducible representation of $\mathbf{G}^F$, occurring in some Deligne-Lusztig representation, $R^\theta_w$. Then
\[|\langle R^\theta_w,\rho\rangle|\leq|\mathbf{W}|^{\frac 12}.\]
\end{proposition}
\begin{proof}
Write $R^\theta_w$ as a $\mathbb{Z}$-linear combination of irreducible representations:
\[R^\theta_w=e_1\rho_1+\dots+e_k\rho_k,\]
where $\rho_1=\rho$. Then
\[e_1^2+\dots+e_k^2=\langle R^\theta_w,R^\theta_w\rangle.\]
By Proposition \ref{inner product bound}, the inner product of two Deligne-Lusztig representations is bounded above by $|\mathbf{W}|$. Therefore, we have $\langle R^\theta_w,\rho\rangle^2=e_1^2\leq |\mathbf{W}|$, and the result follows.
\end{proof}
\begin{proposition}\label{geometric conjugacy class bound}\cite[Lemma 2.3.5]{geckmalle20}
Fix an $F$-stable maximal torus $\mathbf{T}$ of $\mathbf{G}$ and $\theta\in\Irr(\mathbf{T}^F)$. There are at most $|\mathbf{W}|$ choices of $\theta'\in\Irr(\mathbf{T}^F)$ for which $(\mathbf{T},\theta')$ is geometrically conjugate to $(\mathbf{T},\theta)$.
\end{proposition}
\begin{proof}
Let $(\mathbf{T},\theta),(\mathbf{T},\theta')$ be geometrically conjugate. Choose $d\geq 0$ and $g\in\mathbf{G}^{F^d}$ such that ${^g}\mathbf{T}=\mathbf{T}$ and $g$ maps $\theta\circ N_{F^d/F}$ to $\theta'\circ N_{F^d/F}$ as in Definition \ref{geometrically conjugate}. Note that $g\in N_{\mathbf{G}}(\mathbf{T})^{F^d}$, so we can write $g=ms$, where $m$ is in some fixed set of left coset representatives for $\mathbf{T}^{F^d}$ in $N_{\mathbf{G}}(\mathbf{T})^{F^d}$ and $s\in \mathbf{T}^{F^d}$. We see that ${\theta\circ N_{F^d/F}(g^{-1}tg)=\theta\circ N_{F^d/F}(m^{-1}tm)}$, therefore $m$ also transforms $\theta\circ N_{F^d/F}$ to $\theta'\circ N_{F^d/F}$. Thus, using Lemma \ref{frobenius factoring} we get at most
\[|N_{\mathbf{G}}(\mathbf{T})^{F^d}/\mathbf{T}^{F^d}|=|\mathbf{W}^{F^d}|\]
many $\theta'$ such that $(\mathbf{T},\theta)$ is geometrically conjugate to $(\mathbf{T},\theta')$.
\end{proof}
Given $\mathbf{G}$, a maximal torus $\mathbf{T}$ of $\mathbf{G}$ and a Frobenius map $F$, consider the induced action of $F$ on the set of roots, $X=\Hom(\mathbf{T},\mathbb{G}_m)$. Following \cite[\S 1.18]{carter93}, take $\delta\in\mathbb{N}$ to be minimal such that $F^\delta=kI$ for some integer $k\geq 2$. Now define $q$ to be such that $k=q^\delta$.
\begin{proposition}[Deligne-Lusztig representation dimensions are polynomial]
Let $\mathbf{G}$ be reductive with Frobenius map $F$ and $F$-stable maximally split torus $\mathbf{T}_0$. For every $w\in\mathbf{W}$, there exists a polynomial $f_w$ such that for every $d\geq 1$ and all $\theta\in\Irr(\mathbf{T}_0[w,F^d])$, $R^\theta_w(1)=f_w(q^d)$, where $q$ is the parameter associated to $F$ as before.
\end{proposition}
\begin{proof}
By \cite[\S 2.9]{carter93}, we can write
\[|\mathbf{G}^F|=|(\mathbf{Z}^\circ)^F|q^N\prod_J(q^{|J|}-1)\sum_{w\in\mathbf{W}^F}q^{l(w)},\]
where the product runs over orbits $J$ of simple roots of $\mathbf{G}$ and $l(w)$ is the minimal length of an expression of $w$ as a product of the generators $s_i$ of $\mathbf{W}$, and $\mathbf{Z}^\circ$ is the connected component of the centre of $\mathbf{G}$.
\par By \cite[Proposition 3.3.8]{carter93},
\[|\mathbf{T}_0[w]|=|(\mathbf{Z}^\circ)^F|\chi(q)\]
where $\chi$ is a polynomial depending only on $\mathbf{G}$ and $F$. Since $p$ does not divide $|\mathbf{T}_0[w]|$ \cite[p.83]{carter93}, we see that \[[\mathbf{G}^F:\mathbf{T}_0[w]]_{p'}=\frac{\prod_J(q^{|J|}-1)\sum_{w\in\mathbf{W}^F}q^{l(w)}}{\chi(q)},\]
which is a polynomial in $q$. Now by \cite[Theorem 7.5.1]{carter93}, $R^\theta_w(1)=\pm[\mathbf{G}^F:\mathbf{T}_0[w]]_{p'}$, where the sign depends only on $\mathbf{G}$ and $w$.
\end{proof}
\par Suppose that $\rho\in\Irr(\mathbf{G}^F)$. Then writing the regular representation of $\mathbf{G}^F$ in terms of Deligne-Lusztig representations by Proposition \ref{uniformity of regular representation}, we obtain
\begin{align*}
\deg\rho&=\langle\chi_\mathrm{reg},\rho\rangle\\
&=\frac 1{|\mathbf{W}|}\sum_{w\in\mathbf{W}}\sum_{\theta\in\Irr(\mathbf{T}_0[w])}R^\theta_w(1)\langle R^\theta_w,\rho\rangle\\
&=\frac 1{|\mathbf{W}|}\sum_{w\in\mathbf{W}}\sum_{\theta\in\Irr(\mathbf{T}_0[w])}\langle R^\theta_w,\rho\rangle f_w(q).
\end{align*}
By Proposition \ref{geometric conjugacy class bound}, the inner sum has at most $|\mathbf{W}|$ non-zero terms, and for each $\theta$, we have ${|\langle R^\theta_w,\rho\rangle|\leq |\mathbf{W}|^{\frac 12}}$. Therefore, the finite set of polynomials
\[R=\left\{\frac 1{|\mathbf{W}|}\sum_{w\in\mathbf{W}}a_wf_w\middle|a_w\in\mathbb{Z},|a_w|\leq|\mathbf{W}|^{\frac 32}\right\}\]
will give all the dimensions of the irreducible representations of $\mathbf{G}^F$. This completes the proof of Theorem \ref{weak polynomial for lie type}.

\section{Proof of the Main Theorem}\label{Proof of the Main Theorem}
Tao \cite{tao} defined the notion of complexity of a variety. We will use a notion of Hadas \cite[Definition 3.1.1]{hadas24} which is valid for varieties over fields which may not be algebraically closed:
\begin{definition}
Let $V\subseteq\mathbb{A}^n$ be an affine variety over an algebraically closed field. The \textbf{complexity} of $V$ is the smallest integer $M\geq n$ such that the vanishing ideal of $V$ is generated by at most $M$ polynomials, all of degree at most $M$. If $\mathbb{F}$ is any field and $V\subseteq\mathbb{A}^n(\mathbb{F})$ is a variety with an $\mathbb{F}$-embedding $\imath:\mathbb{A}^n(\mathbb{F})\to\mathbb{A}^n(\overline{\mathbb{F}})$, define the complexity of $V$ to be the complexity of the image, $\imath(V)\subseteq\mathbb{A}^n(\overline{\mathbb{F}})$.
\end{definition}
Following \cite{hadas24}, let $\mathbb{F}$ be a finite field and define $\mathcal{A}(n',M_{\mathrm{dim}},M_{\mathrm{cmp}},\mathbb{F})$ to be the set of all algebraic subgroups $G$ of $\mathrm{GL}_{n'}(\overline{\mathbb{F}})$ with dimension at most $M_{\mathrm{dim}}$ and complexity at most $M_{\mathrm{cmp}}$ with an $\mathbb{F}$-rational structure. We shall abuse notation and write $G(\mathbb{F})$ in place of $G_0(\mathbb{F})$, where $G=G_0\times_{\mathbb{F}}\overline{\mathbb{F}}$ and $G_0$ is the algebraic group over $\mathbb{F}$ determining the $\mathbb{F}$-rational structure of~$G$.
\par We now list some results necessary for the proof.
\begin{lemma}\label{H^2 for reductive groups}\cite[Lemma 4.2.1, Corollary 4.2.6]{hadas24}
For every $m\in\mathbb{N}$, there exists ${C_{H^2}(m)>0}$ such that for every prime $p>C_{H^2}(m)$, every finite field $\mathbb{F}$ of characteristic $p$ and every connected reductive $G$ defined over $\mathbb{F}$ of dimension at most~$m$, ${H^2(G(\mathbb{F}),\mu_{p^\infty})=1}$.
Write $U=R_u(G)$ and let $p>C_{H^2}(1+\dim G)$. Then every $\rho\in\Irr(G(\mathbb{F}))$ whose restriction to $U(\mathbb{F})$ is isotypic is a tensor product of a representation of $(G/U)(\mathbb{F})$ and a representation of $U(\mathbb{F})$.
\end{lemma}
\begin{lemma}\label{unipotent groups}\cite[Corollary 4.3.2]{hadas24}
For all $n',M\geq 1$, there exists $C_\mathrm{uni}(n',M)\geq 1$ such that for every finite field, $\mathbb{F}$, of characteristic greater than $C_\mathrm{uni}(n',M)$, and every unipotent subgroup, $U$, of $\mathrm{GL}_{n'}(\overline{\mathbb{F}})$ which is defined over $\mathbb{F}$ and has complexity at most $M$,
\[\{\dim\rho\mid\rho\in\Irr(U(\mathbb{F}))\}\subseteq\{|\mathbb{F}|^k\mid 0\leq k\leq (n')^2\}.\]
\end{lemma}
\begin{lemma}\label{stabiliser is algebraic subgroup}\cite[Lemma 4.3.4]{hadas24}
Let $n',M\in\mathbb{N}$. There exist $C_\mathrm{noniso}(n',M)>0$ and $C_\mathrm{stab}(n',M)>0$ such that for every prime $p>C_\mathrm{stab}(n',M)$, every finite field $\mathbb{F}$ of characteristic $p$, every connected algebraic subgroup $G$ of $GL_{n'}(\overline{\mathbb{F}})$ defined over $\mathbb{F}$ and of complexity at most $M$, and every $\chi\in\Irr(R_u(G))(\mathbb{F})$, there exists an algebraic subgroup $K$ of $G$, defined over $\mathbb{F}$ and of complexity at most $C_\mathrm{noniso}(n',M)$, such that $\Stab_{G(\mathbb{F})}(\chi)=K(\mathbb{F})$.
\end{lemma}
\begin{lemma}\label{reductive groups}\cite[Lemma 4.4.2]{hadas24}
For every integer $M\geq 0$, there exists $C_{\mathrm{red}}(M)\geq 0$ such that for every algebraically closed field $k$, there exist at most $C_{\mathrm{red}}(M)$ reductive groups over $k$ of dimension at most $M$.
\end{lemma}
Given an algebraic group $G$ and an algebraic subgroup $K$ both defined over $\mathbb{F}_p$ for some prime $p$, we will want to describe the behaviour of $[G(\mathbb{F}_{p^d}):K(\mathbb{F}_{p^d})]$ as $d$ varies over $\mathbb{N}$. This does not necessarily define a polynomial function, but one with a weaker property which we will define. Higman \cite{higman60} defined the notion of a PORC function, however we modify the definition for convenience here.
\begin{definition}
Let $q$ be a prime power. A function $f:\{q^d\mid d\in \mathbb{N}\}\to\mathbb{Z}$ is \textbf{$q$-PORC} (polynomial on residue classes) if there exist $N\in\mathbb{N}$ and polynomials $g_0,\dots,g_{N-1}\in\mathbb{Q}[x]$ such that for all $d,r\geq 1$ with $n\equiv r\bmod N$, ${f(q^d)=g_r(q^d)}$. We call $N$ the \textbf{period} of $f$ and $g_0,\dots,g_{N-1}$ the \textbf{constituents} of~$f$.
\end{definition}
\begin{proposition}\label{brion peyre}\cite[Theorem 1.2, Remark 5.1(ii)]{brionpeyre08}
Let $G$ be an algebraic group defined over a finite field $\mathbb{F}_q$ and let $X$ be a variety with homogeneous $G$-action. Then the function
\[q^d\mapsto X(\mathbb{F}_{q^d})\]
is a $q$-PORC function. Its period is bounded above by a constant that depends only on the rank of $G$.
\end{proposition}
The statement in \cite{brionpeyre08} states further that the constituents of this $q$-PORC function have integer coefficients.
\par The following result states that for $p$ prime, the values taken by any set of $p$-~PORC functions whose period is bounded above are also taken by some finite set of polynomials.
\begin{lemma}[$p$-PORC Lemma]\label{p-PORC lemma}
Let $p$ be a prime and $\{f_i\}_{i\in I}$ be a set of $p$-PORC functions $\{p^d\mid d\geq 1\}\to\mathbb{Z}$ such that there exists $N>0$ and for all $i\in I$, the period of $f_i$ is at most $N$. Suppose further that there exists $C>0$ such that for all $d\geq 1$, $|\{f_i(p^d)\mid i\in I\}|\leq C$. Then there exists a finite set of polynomials $A\subseteq\mathbb{Q}[x]$ such that for every $i\in I,d\geq 1$, there exists $g\in A$ such that $f_i(p^d)=g(p^d)$.
\end{lemma}
\begin{proof}
Without loss of generality, let $N$ be the least common multiple of the periods of the $f_i$. Fix any residue class $m$ modulo $N$. Then there exist polynomials $g_{m,i}$, for $i\in I$ such that for all $d\equiv m\bmod N$, $f_i(p^d)=g_{m,i}(p^d)$. Let
\[C_m=\max_{d\equiv m\bmod N}|\{f_i(p^d)\mid i\in I\}|.\]
Choose $C_m$ distinct polynomials from the set $\{g_{m,i}\mid i\in I\}$, say ${A_m=\{g_{m,j}|j\in J\}}$ where $J\subseteq I$ and $|J|=C_m$. Then there exists $d_m>0$ such that for all $j_1,j_2\in J$ with $j_1\neq j_2$ and all $d\geq d_m$, we have $g_{m,j_1}(p^d)\neq g_{m,j_2}(p^d)$.\\
Write $d_0=\max_m d_m$ and $A^+=\bigcup_{m\in\mathbb{Z}/N\mathbb{Z}}A_m$. Then for all $d\geq d_0$ and all $i\in I$, there exists $g\in A$ such that $f_i(p^d)=g(p^d)$. Now put
\[A=A^+\cup\bigcup_{i\in I} \{f_i(p^d)\mid d<d_0\},\]
which is still a finite set by assumption.
\end{proof}
\noindent We also remark that the quotient of two $p$-PORC functions is again $p$-PORC.
By \cite[Theorem 4.0.2]{hadas24}, there exists a bound, uniform in $|\mathbb{F}|$, on the cardinality of the set
\[\bigcup_{G\in\mathcal{A}(n',M_{\mathrm{dim}},M_{\mathrm{cmp}},\mathbb{F})}\mathrm{dimirr}(G(\mathbb{F})).\]
The following result states further that the dimensions of the representations of $G(\mathbb{F})$ are given by evaluating finitely many polynomials at $|\mathbb{F}|$. It is the main step in proving Theorem \ref{weak polynomial for large p}.
\begin{theorem}\label{hadas polynomial version}
For all $n',M_{\mathrm{dim}},M_{\mathrm{cmp}}\in\mathbb{N}$, there exists $C>0$ such that for every prime $p>C$, there exists a finite set of polynomials $R_{n',M_{\mathrm{dim}},M_{\mathrm{cmp}}}\subseteq\mathbb{Q}[x]$ such that for every finite field $\mathbb{F}$ of characteristic $p$,
\begin{equation}\label{main equation}
\bigcup_{G\in\mathcal{A}(n',M_{\mathrm{dim}},M_{\mathrm{cmp}},\mathbb{F})}\mathrm{dimirr}(G(\mathbb{F}))\subseteq\left\{m(|\mathbb{F}|)\middle| m\in R_{n',M_{\mathrm{dim}},M_{\mathrm{cmp}}}\right\}.
\end{equation}
\end{theorem}
The proof follows that of \cite[Theorem 4.0.2]{hadas24}, with the difference that we employ the argument of Section \ref{Finite Groups of Lie Type} to deal with the reductive case, and we argue at each step that the dimensions of the representations constructed are in fact given by finitely many polynomials.
\begin{proof}
We argue by induction on $M_{\mathrm{dim}}$. For the base case, $M_{\mathrm{dim}}=0$, every zero-dimensional closed subgroup of $GL_{n'}(\overline{\mathbb{F}})$ is a discrete finite group, hence reductive. For the moment, fix a prime $p$. By Theorem \ref{weak polynomial for lie type}, for each $G\in\mathcal{A}(n',0,M_{\mathrm{cmp}},\mathbb{F})$ we can choose a finite set $R_G$ of polynomials such that for every finite field $\mathbb{F}$ of characteristic $p$,
\[\mathrm{dimirr}(G(\mathbb{F}))\subseteq\left\{m(|\mathbb{F}|)\middle| m\in R_G\right\}.\]
By Lemma \ref{reductive groups}, the set $\mathcal{A}(n',0,M_{\mathrm{cmp}},\mathbb{F})$ is finite with cardinality at most $C_\mathrm{red}(0)$, so we can let
\[R_{n',0,M_\mathrm{cmp}}=\bigcup_{G\in\mathcal{A}(n',0,M_\mathrm{cmp},\mathbb{F})}R_G.\]
We proceed with the inductive step.\\
\textbf{Step 1: Reduce to those $G$ which are connected, and neither reductive nor unipotent.}\\
Suppose for induction that the theorem is true for all dimensions strictly less than~$M_{\mathrm{dim}}$. By \cite[Lemma 11]{tao}, there exists $C_{\mathrm{ic}}(M_{\mathrm{cmp}})>0$ such that for every finite field~$\mathbb{F}$ and every $G\in\mathcal{A}(n',M_\mathrm{dim},M_\mathrm{cmp},\mathbb{F})$, we have \[[G(\mathbb{F}):G^\circ(\mathbb{F})]=|(G/G^\circ)(\mathbb{F})|\leq C_{\mathrm{ic}}(M_{\mathrm{cmp}}),\] and that the dimensions of the representations of $G(\mathbb{F})$ are sums of dimensions of representations of $G^\circ(\mathbb{F})$ with at most $C_{\mathrm{ic}}(M_{\mathrm{cmp}})$ terms. Therefore, it is sufficient to show expression (\ref{main equation}) modified to take the union over all ${G\in\mathcal{A}(n',M_\mathrm{dim},M_\mathrm{cmp},\mathbb{F})}$ which are connected.
\par As in the base case, we see that by Lemma \ref{reductive groups} there are at most $C_\mathrm{red}(M_\mathrm{dim})$ reductive groups $G\in\mathcal{A}(n',M_\mathrm{dim},M_\mathrm{cmp},\mathbb{F})$. We can use Theorem \ref{weak polynomial for lie type} to show that there exists a finite set $T_{\mathrm{red},M_{\mathrm{dim}}}\subseteq\mathbb{Q}[x]$ such that for every prime $p$ and every finite field $\mathbb{F}$ of characteristic $p$,
\[\bigcup_{\substack{G\in\mathcal{A}(n',M_{\mathrm{dim}},M_{\mathrm{cmp}},\mathbb{F})\\G \text{ reductive}}}\mathrm{dimirr}(G(\mathbb{F}))\subseteq\left\{m(|\mathbb{F}|)\middle| m\in T_{\mathrm{red},M_{\mathrm{dim}}}\right\}.\]
\par Consider the case where $G$ is unipotent. By Lemma \ref{unipotent groups}, there exists $C_\mathrm{uni}>0$ such that for all primes $p>C_\mathrm{uni}$ and every finite field $\mathbb{F}$ of characteristic $p$, there is a finite set $T_{\mathrm{uni},M_{\mathrm{dim}}}\subseteq\mathbb{Q}[x]$ such that
\[\bigcup_{\substack{G\in\mathcal{A}(n',M_{\mathrm{dim}},M_{\mathrm{cmp}},\mathbb{F})\\G \text{ unipotent}}}\mathrm{dimirr}(G(\mathbb{F}))\subseteq\left\{m(|\mathbb{F}|)\middle| m\in T_{\mathrm{uni},M_{\mathrm{dim}}}\right\}.\]
In fact, Lemma \ref{unipotent groups} implies that we can choose $T_{\mathrm{uni},M_\mathrm{dim}}\subseteq\{1,x,\dots,x^{(n')^2}\}$.\\
\textbf{Step 2:}\\
We can assume that $p$ is large in relation to $n',M_\mathrm{dim}$ and $M_\mathrm{cmp}$, therefore assume from now on that $p>\max\{C_{H^2}(1+M_\mathrm{dim}),C_\mathrm{stab}(n',M_\mathrm{cmp})\}$. Let $U=R_u(G)$ and suppose now that $U\neq 1,G$. Suppose $\rho\in\Irr(G(\mathbb{F}))$.
\par If $\Res^{G(\mathbb{F})}_{U(\mathbb{F})}\rho$ is isotypic, then since $p>C_{H^2}(1+M_\mathrm{dim})$, $\rho$ is a tensor product of a representation of $U(\mathbb{F})$ with a representation of $G(\mathbb{F})/U(\mathbb{F})=(G/U)(\mathbb{F})$ by Lemma \ref{H^2 for reductive groups}).
\par If $\Res^{G(\mathbb{F})}_{U(\mathbb{F})}\rho$ is not isotypic, then since $p>C_\mathrm{stab}(n',M_\mathrm{cmp})$, $\rho=\Ind^{G(\mathbb{F})}_{K(\mathbb{F})}\psi$ where $K$ is a proper algebraic subgroup of $G$ (so $\dim K<\dim G$), $\psi\in\Irr(K(\mathbb{F}))$ and $\mathrm{cmp}(K)\leq C_{\mathrm{noniso}}(n',M_{\mathrm{cmp}})$ by Lemma \ref{stabiliser is algebraic subgroup}.\\
\textbf{Step 2(a): The isotypic case}\\
Define $T_{\mathrm{iso},M_{\mathrm{dim}}}=\{m_1(x)m_2(x)|m_1\in T_{\mathrm{uni},M_\mathrm{dim}},m_2\in T_{\mathrm{red},M_{\mathrm{dim}}}\}$. Then since $U$ is unipotent and $G/U$ is reductive, for every finite field $\mathbb{F}$ of characteristic $p$,
\begin{align*}\bigcup_{\substack{G\in\mathcal{A}(n',M_\mathrm{dim},M_\mathrm{cmp},\mathbb{F})\\R_u(G)\neq 1,G}}&\{\dim\rho\mid\rho\in\Irr(G(\mathbb{F})),\Res^{G(\mathbb{F})}_{U(\mathbb{F})}\rho\text{ is isotypic}\}\\&\subseteq\{m(|\mathbb{F}|)\mid m\in T_{\mathrm{iso},M_{\mathrm{dim}}}\}.\end{align*}
\textbf{Step 2(b): The non-isotypic case}\\
Assume for induction that for all $M_{\mathrm{cmp}}'\in\mathbb{N}$, there exists a finite set of polynomials $T_{M_{\mathrm{dim}}-1,M_{\mathrm{cmp}}'}$ such that for every $\mathbb{F}$ of characteristic $p$,
\[\bigcup_{G\in\mathcal{A}(n',M_{\mathrm{dim}}-1,M_{\mathrm{cmp}}',\mathbb{F})}\mathrm{dimirr}(G(\mathbb{F}))\subseteq\left\{m(|\mathbb{F}|)\middle|m\in T_{M_{\mathrm{dim}}-1,M_{\mathrm{cmp}}'}\right\}.\]
We want to find a finite set $T_{\mathrm{noniso},M_\mathrm{dim}}\subseteq\mathbb{Q}[x]$ such that for every finite field $\mathbb{F}$ of characteristic $p$,
\begin{align*}\bigcup_{\substack{G\in\mathcal{A}(n',M_\mathrm{dim},M_\mathrm{cmp},\mathbb{F})\\R_u(G)\neq 1,G}}&\{\dim\rho\mid\rho\in\Irr(G(\mathbb{F})),\Res^{G(\mathbb{F})}_{U(\mathbb{F})}\rho\text{ is non-isotypic}\}\\&\subseteq\{m(|\mathbb{F}|)\mid m\in T_{\mathrm{noniso},M_{\mathrm{dim}}}\}.\end{align*}
Suppose that $\rho$ has non-isotypic restriction to $U(\mathbb{F})$, hence $\rho=\mathrm{Ind}^{G(\mathbb{F})}_{K(\mathbb{F})}\psi$, for some proper algebraic subgroup $K\leq G$ defined over $\mathbb{F}$. Then \[\dim\rho=[G(\mathbb{F}):K(\mathbb{F})]\dim\psi.\]
We show that the index $[G(\mathbb{F}):K(\mathbb{F})]$ is given by finitely many polynomials: the coset space $G(\overline{\mathbb{F}})/K^\circ(\overline{\mathbb{F}})=(G/K^\circ)(\overline{\mathbb{F}})$ is a homogeneous space for the algebraic group $G(\overline{\mathbb{F}})$, therefore the number of $\mathbb{F}$-points, $f_{G/K^\circ}(|\mathbb{F}|):=|(G/K^\circ)(\mathbb{F})|$, is given by a $p$-PORC function by Proposition \ref{brion peyre}.
Further, we can apply the same argument to show that $f_{K/K^\circ}(|\mathbb{F}|):=|(K/K^\circ)(\mathbb{F})|$ is $p$-PORC, therefore
\[[G(\mathbb{F}):K(\mathbb{F})]=\frac{|G(\mathbb{F})/K^\circ(\mathbb{F})|}{|K(\mathbb{F})/K^\circ(\mathbb{F})|}\]
is also a $p$-PORC function of $|\mathbb{F}|$.
\par There may be infinitely many pairs $(G,K)$ that appear in this process; let them be indexed by $\{(G_i,K_i)\mid i\in I\}$. However, by \cite[Lemma 4.5.1]{hadas24}, there exist $C_{\mathrm{GC},G},C_{\mathrm{GC},K}>0$ such that for all $j\geq 1$, the sets $A_j:=\{|G_i(\mathbb{F}_{p^j})|\mid i\in I\}$ and $B_j:=\{|K_i(\mathbb{F}_{p^j})|\mid i\in I\}$ are finite, with $|A_j|\leq C_{\mathrm{GC},G},|B_j|\leq C_{\mathrm{GC},K}$. Therefore, for all $j\geq 1$,
\[|\{[G_i(\mathbb{F}_{p^j}):K_i(\mathbb{F}_{p^j})]\mid i\in I\}|\leq C_{\mathrm{GC},G}C_{\mathrm{GC},K}=:C_\mathrm{GC}.\]
Suppose without loss of generality that this bound is optimal, i.e. there exists $j$ such that $|\{[G_i(\mathbb{F}_{p^j}):K_i(\mathbb{F}_{p^j})]\mid i\in I\}|=C_\mathrm{GC}$. By Lemma \ref{brion peyre}, there exist bounds on the periods of $f_{G/K^\circ}$ and $f_{K/K^\circ}$ depending only on the ranks of $G$ and $K$ (both bounded above by $M_\mathrm{dim}$), therefore we can apply Lemma \ref{p-PORC lemma}. We conclude that there exists a finite set of polynomials $T_\mathrm{index}\subseteq\mathbb{Q}[x]$ with the property that for all $j\geq 1$,
\[\{[G_i(\mathbb{F}_{p^j}):K_i(\mathbb{F}_{p^j})]\mid i\in I\}\subseteq\{m(p^j)\mid m\in T_\mathrm{index}\}.\]
The representations in the non-isotypic case are now given by the finitely many polynomials
\[T_{\mathrm{noniso},M_{\mathrm{dim}}}=\left\{m_1(x)m_2(x)\middle|m_1\in T_\mathrm{index},m_2\in T_{M_{\mathrm{dim}}-1,C_{\mathrm{noniso}}(n',M_{\mathrm{cmp}})}\right\}.\]
\textbf{Step 3: Conclusion}\\
Since we already have defined $R_{n',0,M_{\mathrm{cmp}}}$, define $R_{n',M_{\mathrm{dim}},M_{\mathrm{cmp}}}$ inductively by:
\[R_{n',M_{\mathrm{dim}},M_{\mathrm{cmp}}}:=R_{n',M_{\mathrm{dim}}-1,M_{\mathrm{cmp}}}\cup T_{\mathrm{red},M_{\mathrm{dim}}}\cup T_{\mathrm{uni},M_{\mathrm{dim}}}\cup T_{\mathrm{iso},M_{\mathrm{dim}}}\cup T_{\mathrm{noniso},M_{\mathrm{dim}}}.\]
By the construction in the inductive step, $R_{n',M_{\mathrm{dim}},M_{\mathrm{cmp}}}$ has the required properties in the statement of the theorem.
\end{proof}
We can now prove a weak polynomial result for unramified extensions of a fixed base field:
\begin{corollary}\label{Weak Onn for towers}
Let $G$ be a smooth affine group scheme over $\mathbb{Z}$ and $r\geq 1$ be an integer. There exists $C>0$ such that for every $p>C$ and every non-Archimedean local field $F$ with finite residue field of characteristic $p$, there exists a finite set of polynomials $R\subseteq \mathbb{Q}[x]$ such that for every $\mathfrak{o}$ which is the valuation ring of some finite unramified extension of $F$,
\[\mathrm{dimirr}(G(\mathfrak{o}_r))\subseteq\left\{m(|\mathfrak{o}_1|)\middle|m\in R\right\}.\]
\end{corollary}
\begin{proof}
Let $\mathscr{G}=\mathrm{Gr}_{r-1}^{\mathfrak{o}}(G\times_\mathbb{Z}\mathfrak{o}_r)$, an algebraic group over $\mathbb{F}_q$, where $q=|\mathfrak{o}_1|$. Then $G(\mathfrak{o}_r)\cong \mathscr{G}(\mathbb{F}_q)$. By Theorem \ref{greenberg functor and extensions of o}, all such $\mathscr{G}$ are given by extension of scalars from $\mathscr{G}_0:=\mathrm{Gr}_{r-1}^{\mathfrak{o}_F}(G)$. Therefore, all such $\mathscr{G}$ have the same extension of scalars to an algebraic group over $\overline{\mathbb{F}_p}$, which is contained in some $\mathcal{A}(n',d,c,\mathbb{F}_q)$. The result now follows from Theorem \ref{hadas polynomial version}.
\end{proof}
Note that the Greenberg functor does not preserve the property of group schemes being reductive, which is why we have to treat the non-reductive case in Theorem \ref{hadas polynomial version}.
\par With this in mind, we would like to extend the result to all non-Archimedean local fields, $F$, with finite residue field. It is already known that such $F$ must be isomorphic as a topological field to a finite extension of either $\mathbb{Q}_p$ or $\mathbb{F}_p((t))$. As it turns out, we can restrict our attention to the case where $F_0$ has ramification index less than $r$ over $\mathbb{Q}_p$ or $\mathbb{F}_p((t))$.
\begin{proposition}
Let $F$ be a finite extension of $\mathbb{Q}_p$ or $\mathbb{F}_p((t))$ with residue cardinality $q$ and ramification index $e\geq r$. Then $\mathfrak{o}_F/\mathfrak{p}_F^r\cong\mathbb{F}_q[[t]]/(t^r)$.
\end{proposition}
\begin{proof}
Let $L$ be the maximal unramified extension of $\mathbb{Q}_p$ contained in $F_0$. Write $\mathfrak{p}_L,\mathfrak{p}_F$ for the primes in $L,F$ respectively. We have natural maps
\[\mathbb{F}_q=\mathfrak{o}_{L}/\mathfrak{p}_L\to\mathfrak{o}_F/\mathfrak{p}_F^e\twoheadrightarrow\mathfrak{o}_F/\mathfrak{p}_F^r.\]
The first map is well-defined because $\mathfrak{p}_F^e=\mathfrak{p}_L\mathfrak{o}_F\supseteq\mathfrak{p}_L$, and the second natural map is well-defined because $e\geq r$. Therefore, $\mathfrak{o}_F/\mathfrak{p}_F^r$ is naturally an $\mathbb{F}_q$-algebra, and we can extend to a map
\[\mathbb{F}_q[[t]]\to\mathfrak{o}_F/\mathfrak{p}_F^r\]
by mapping $t\mapsto\varpi$. This latter map is surjective with kernel $(t^r)$.
\end{proof}
We can therefore ignore local fields of ramification index greater or equal to $r$ over $\mathbb{Q}_p$ or $\mathbb{F}_p((t))$.
\begin{proof}[Proof of Theorem \ref{weak polynomial for large p}]
Let $K$ be either $\mathbb{Q}_p$ or $\mathbb{F}_p((t))$. For the moment, fix a finite extension $F/K$ of ramification index $e<r$. We can assume $p$ is large enough so that $p$ does not divide $e$, so $F$ is tamely ramified over $K$. Write $F'=F^\mathrm{unr}=K(\sqrt[e]{p})$ and $\mathscr{G}_{\mathfrak{o}_F,r}=\Gr^{\mathfrak{o}_F}_{r-1}(G_{\mathfrak{o}_F/\mathfrak{p}_F^r})$. By Theorem \ref{greenberg functor and extensions of o} applied to the extension $\mathfrak{o}_{F'}/\mathfrak{o}_F$, we get
\begin{align*}
\mathscr{G}_{\mathfrak{o}_F,r}\times_{\mathbb{F}_{p^f}}\overline{\mathbb{F}_p}&\cong\Gr^{\mathfrak{o}_{F'}}_{r-1}(G_{\mathfrak{o}_F/\mathfrak{p}_F^r}\times_{\mathfrak{o}_F/\mathfrak{p}_F^r}\mathfrak{o}_{F'}/\mathfrak{p}_{F'}^r)\\
&\cong\Gr^{\mathfrak{o}_{F'}}_{r-1}(G_{\mathfrak{o}_{F'}/\mathfrak{p}_{F'}^r}),
\end{align*}
which does not depend on $F$. Therefore, if $F,F'$ are extensions of $K$ with ramification index $e$, then $\mathscr{G}_{\mathfrak{o}_F,r}$ and $\mathscr{G}_{\mathfrak{o}_{F'},r}$ have the same extension of scalars $\mathscr{H}$ to ${F^\mathrm{unr}=(F')^\mathrm{unr}=K^\mathrm{unr}(\sqrt[e]{p})}$. We can now repeat the argument of Corollary \ref{Weak Onn for towers}: since $\mathscr{H}$ is contained in some $\mathcal{A}(n',d,c,\mathbb{F}_q)$, there is a finite set $R_e\subseteq\mathbb{Q}[x]$ such that for every $\mathfrak{o}$ which is the valuation ring of an extension of $\mathbb{Q}_p$ or $\mathbb{F}_p((t))$ of ramification index $e$,
\[\dimirr(G(\mathfrak{o}_r))\subseteq\{m(|\mathfrak{o}_1|)\mid m\in R_e\}.\]
Let $R=\cup_{e=1}^{r-1}R_e$ for the result.
\end{proof}

Acknowledgements: The author was supported by the HIMR/UKRI Additional Funding Programme for Mathematical Sciences, EP/V521917/1. The author also gratefully acknowledges the supervision of Prof. Alexander Stasinski and Dr. Jack Shotton, and Daniel Funck and Paul Helminck for many enlightening conversations.
\bibliographystyle{plain}
\bibliography{Polynomial_Paper.bib}

\begin{thebibliography}{10}

\bibitem{aizenbudavni22}
Nir Avni and Avraham Aizenbud.
\newblock Bounds on multiplicities of symmetric pairs of finite groups, 2022.
\newblock arXiv:2202.12217 [math.RT], 2022.

\bibitem{bertapellegonzalezaviles18}
Alessandra Bertapelle and Cristian~D. Gonz\'{a}lez-Avil\'{e}s.
\newblock The {G}reenberg functor revisited.
\newblock {\em Eur. J. Math.}, 4(4):1340--1389, 2018.

\bibitem{brionpeyre08}
Michel Brion and Emmanuel Peyre.
\newblock Counting points of homogeneous varieties over finite fields.
\newblock {\em J. Reine Angew. Math.}, 645:105--124, 2010.

\bibitem{carter93}
Roger~W. Carter.
\newblock {\em Finite groups of {L}ie type: Conjugacy classes and complex
  characters}.
\newblock Wiley Classics Library. John Wiley \& Sons, Ltd., Chichester, 1993.
\newblock Reprint of the 1985 original, A Wiley-Interscience Publication.

\bibitem{clifford37}
A.~H. Clifford.
\newblock Representations induced in an invariant subgroup.
\newblock {\em Ann. of Math. (2)}, 38(3):533--550, 1937.

\bibitem{delignelusztig76}
P.~Deligne and G.~Lusztig.
\newblock Representations of reductive groups over finite fields.
\newblock {\em Ann. of Math. (2)}, 103(1):103--161, 1976.

\bibitem{geck17}
Meinolf Geck.
\newblock A first guide to the character theory of finite groups of lie type.
\newblock arXiv:1705.05083 [math.RT], 2017.

\bibitem{geck11}
Meinolf Geck.
\newblock Remarks on modular representations of finite groups of lie type in
  non-defining characteristic.
\newblock arXiv:1107.0296 [math.RT], 2011.

\bibitem{geckmalle20}
Meinolf Geck and Gunter Malle.
\newblock {\em The Character Theory of Finite Groups of Lie Type: A Guided
  Tour}.
\newblock Cambridge Studies in Advanced Mathematics. Cambridge University
  Press, 2020.

\bibitem{green}
J.~A. Green.
\newblock The characters of the finite general linear groups.
\newblock {\em Trans. Amer. Math. Soc.}, 80:402--447, 1955.

\bibitem{greenberg61}
Marvin~J. Greenberg.
\newblock Schemata over local rings.
\newblock {\em Ann. of Math. (2)}, 73:624--648, 1961.

\bibitem{greenberg63}
Marvin~J. Greenberg.
\newblock Schemata over local rings. {II}.
\newblock {\em Ann. of Math. (2)}, 78:256--266, 1963.

\bibitem{hadas24}
Itamar Hadas.
\newblock Spectral equivalence of smooth group schemes over principal ideal
  local rings.
\newblock {\em Journal of Algebra}, 2024.

\bibitem{halasipalfy11}
Zolt\'{a}n Halasi and P\'{e}ter~P. P\'{a}lfy.
\newblock The number of conjugacy classes in pattern groups is not a polynomial
  function.
\newblock {\em J. Group Theory}, 14(6):841--854, 2011.

\bibitem{higman60part1}
Graham Higman.
\newblock Enumerating {$p$}-groups. {I}. {I}nequalities.
\newblock {\em Proc. London Math. Soc. (3)}, 10:24--30, 1960.

\bibitem{higman60}
Graham Higman.
\newblock Enumerating {$p$}-groups. {II}. {P}roblems whose solution is {PORC}.
\newblock {\em Proc. London Math. Soc. (3)}, 10:566--582, 1960.

\bibitem{hilljordan}
Gregory Hill.
\newblock A {J}ordan decomposition of representations for
  {$\mathrm{GL}_n(\mathscr{O})$}.
\newblock {\em Comm. Algebra}, 21(10):3529--3543, 1993.

\bibitem{hillnilpotent}
Gregory Hill.
\newblock On the nilpotent representations of {$\mathrm{GL}_n(\mathscr{O})$}.
\newblock {\em Manuscripta Math.}, 82(3-4):293--311, 1994.

\bibitem{hill95}
Gregory Hill.
\newblock Regular elements and regular characters of
  {$\mathrm{GL}_n(\mathscr{O})$}.
\newblock {\em J. Algebra}, 174(2):610--635, 1995.

\bibitem{hillsemisimplecuspidal}
Gregory Hill.
\newblock Semisimple and cuspidal characters of {$\mathrm{GL}_n(\mathscr{O})$}.
\newblock {\em Comm. Algebra}, 23(1):7--25, 1995.

\bibitem{isaacs}
Irving~Martin Isaacs.
\newblock {\em Character theory of finite groups}.
\newblock Pure and applied mathematics (Academic Press) ; 69. Academic Press,
  New York, 1976.

\bibitem{krakovskionnsingla}
Roi Krakovski, Uri Onn, and Pooja Singla.
\newblock Regular characters of groups of type {$A_n$} over discrete valuation
  rings.
\newblock {\em J. Algebra}, 496:116--137, 2018.

\bibitem{nicaisesebag08}
Johannes Nicaise and Julien Sebag.
\newblock Motivic {S}erre invariants and {W}eil restriction.
\newblock {\em J. Algebra}, 319(4):1585--1610, 2008.

\bibitem{onn08}
Uri Onn.
\newblock Representations of automorphism groups of finite
  {$\mathfrak{o}$}-modules of rank two.
\newblock {\em Adv. Math.}, 219(6):2058--2085, 2008.

\bibitem{onnprasadsingla23}
Uri Onn, Amritanshu Prasad, and Pooja Singla.
\newblock Representation zeta functions of arithmetic groups of type ${A}_2$ in
  positive characteristic, 2023.
\newblock arXiv:2308.07073 [math.RT], 2023.

\bibitem{simpsonframe73}
William~A. Simpson and J.~Sutherland Frame.
\newblock The character tables for {${\rm SL}(3,\,q)$}, {${\rm
  SU}(3,\,q\sp{2})$}, {${\rm PSL}(3,\,q)$}, {${\rm PSU}(3,\,q\sp{2})$}.
\newblock {\em Canadian J. Math.}, 25:486--494, 1973.

\bibitem{singla10}
Pooja Singla.
\newblock On representations of general linear groups over principal ideal
  local rings of length two.
\newblock {\em J. Algebra}, 324(9):2543--2563, 2010.

\bibitem{stasinski08}
Alexander Stasinski.
\newblock The smooth representations of {$\mathrm{GL}_2(\mathfrak{o})$}.
\newblock {\em Comm. Algebra}, 37(12):4416--4430, 2009.

\bibitem{stasinskistevens}
Alexander Stasinski and Shaun Stevens.
\newblock The regular representations of {${\rm GL}_N$} over finite local
  principal ideal rings.
\newblock {\em Bull. Lond. Math. Soc.}, 49(6):1066--1084, 2017.

\bibitem{sv-g19}
Alexander Stasinski and Andrea Vera-Gajardo.
\newblock Representations of reductive groups over finite local rings of length
  two.
\newblock {\em J. Algebra}, 525:171--190, 2019.

\bibitem{tao}
Terence Tao.
\newblock Ultralimit analysis, and quantitative algebraic geometry.
\newblock
  https://terrytao.wordpress.com/2010/01/30/the-ultralimit-argument-and-quantitative-algebraic-geometry/.
\newblock Accessed: 2024-05-22.

\end{thebibliography}
\end{document}